\documentclass[a4paper,11pt]{amsart}

\usepackage{amssymb}
\usepackage{amsfonts}
\usepackage{amsmath}
\usepackage[all]{xy}
\usepackage{array}

\theoremstyle{plain}
\newtheorem{theorem}{Theorem}
\newtheorem{lemma}[theorem]{Lemma}
\newtheorem{proposition}[theorem]{Proposition}
\newtheorem*{theoremnn}{Theorem}
\newtheorem{corollary}[theorem]{Corollary}

\theoremstyle{definition}
\newtheorem{definition}{Definition}
\newtheorem{example}{Example}

\theoremstyle{remark}
\newtheorem{remark}{Remark}

\newcommand{\Si}{\mathfrak{S}}
\newcommand{\Ext}{{\mathrm{Ext}}}
\renewcommand{\hom}{{\mathrm{Hom}}}

\newcommand{\op}{\mathrm{op}}
\newcommand{\Fp}{\mathbb{F}_p}
\newcommand{\Fq}{\mathbb{F}_q}
\newcommand{\PP}{\mathcal{P}}
\newcommand{\V}{\mathcal{V}}
\newcommand{\gl}{\mathrm{gl}}

\newcommand{\Tw}{\mathrm{Tw}}
\newcommand{\DD}{\mathbf{D}}
\newcommand{\LL}{\mathbf{L}}
\newcommand{\RR}{\mathbf{R}}
\newcommand{\Gr}{\mathrm{Gr}\,}
\newcommand{\Ch}{\mathrm{Ch}}

\begin{document}

%\title[On the existence of the universal classes]{On the existence of the universal classes for algebraic groups}
\title[A construction of the universal classes]{A construction of the universal classes for algebraic groups with the twisting spectral sequence}
\author{Antoine Touz\'e }

%\authors{A.~Touz\'e
%\thanks{Partially supported by the ANR HGRT (Projet BLAN08-2 338236)}
%\address Universit\'e Paris 13,\\
%LAGA,\\
%CNRS, UMR 7539,\\
%F-93430, Villetaneuse, France
%\email touze@math.univ-paris13.fr
%}

\maketitle

\begin{abstract}

In this article, we adapt some ideas developed by M. Cha{\l}upnik in \cite{Chal3} to the framework of strict polynomial bifunctors. This allows us to get a new proof of the existence of the `universal classes' originally constructed in \cite{TouzeUniv}.
\end{abstract}

\section{Introduction}

Let $\Bbbk$ be a field of prime characteristic $p$, and let $GL_{n,\Bbbk}$ denote the general linear group scheme over $\Bbbk$. In \cite{TouzeUniv}, we exhibited a set of `universal classes' $c[d]$, $d\in\mathbb{N}$, living in the cohomology of $GL_{n,\Bbbk}$. The cohomology we refer to here is the one introduced by Hochschild, which is also called rational cohomology, see \cite[I.4]{Jantzen}. These classes' existence was anticipated by van der Kallen \cite{VdKGross}, and they are one of the key ingredients to prove van der Kallen's conjecture, which is now a theorem:
\begin{theoremnn}[{\cite{TVdK}}]
Let $G$ be a reductive algebraic group scheme over a field $\Bbbk$, and let $A$ be an finitely generated $\Bbbk$-algebra acted on by $G$ via algebra automorphisms. Then the cohomology $H^*(G,A)$ is finitely generated as a graded $\Bbbk$-algebra.  
\end{theoremnn}

The purpose of this article is to give a new proof of the existence of the universal classes $(c[d])_{d\in\mathbb{N}}$. To be more specific, if $V$ is a finite dimensional vector space and $d\ge 1$, the vector space $V^{\otimes d}$ is acted on by the symmetric group $\Si_d$. We denote by $\Gamma^d(V)$ the `$d$-th divided power of $V$', that is the subspace of invariants $(V^{\otimes d})^{\Si_d}$. We also denote by $\mathfrak{gl}_n$ the adjoint representation of $GL_{n,\Bbbk}$ and by $\mathfrak{gl}_n^{(1)}$ the representation obtained by base change along the Frobenius morphism $\Bbbk\to\Bbbk, x\mapsto x^p$. We give in section \ref{sec-proof} a new proof of the following theorem, originally established in \cite[Thm 0.1]{TouzeUniv}.
\begin{theorem}\label{thm-univ}
Let $\Bbbk$ be a field of positive characteristic and let $n\ge p$ be an integer. There are cohomology classes $c[d]\in H^{2d}(GL_{n,\Bbbk},\Gamma^d(\mathfrak{gl}_n^{(1)}))$ such that~:
\begin{enumerate}
\item $c[1]\in H^2(GL_{n,\Bbbk},\mathfrak{gl}_n^{(1)})$ is non zero.
\item If $d\ge 1$ and $\Delta_{(1,\dots,1)} : \Gamma^d(\mathfrak{gl}_n^{(1)})\to (\mathfrak{gl}_n^{(1)})^{\otimes d}$ is the inclusion, then $\Delta_{(1,\dots,1)\,*}c[d]=c[1]^{\cup d}$.
\end{enumerate}
\end{theorem}

The old proof was essentially hand-made: in \cite{TouzeUniv}, we built the universal classes by computing explicit cycles, 
using explicit coresolutions of the representation $\Gamma^d(\mathfrak{gl}_n^{(1)})$. 
To achieve the construction, we introduced two main new combinatorial ingredients: 
the twist compatible category, constructed in \cite[Section 3]{TouzeUniv}, 
and a result on the combinatorics of tensor products of $p$-complexes \cite[Prop 2.4]{TouzeUniv}. 

The new proof uses some ideas developed in \cite{Chal3} to tackle the collapsing conjecture \cite[Conjecture 8.1]{TouzeTroesch}. 
It involves rather different ingredients, namely 
the flexibility of derived categories, the formality phenomenon of \cite[section 4]{TouzeTroesch}, 
the adjunction argument of \cite[section 2]{Chal3} and the twist injectivity \cite{CPS}, \cite[II.10.16]{Jantzen}. 

In the new proof as in the old proof, the natural framework is Franjou and Friedlander's category of strict polynomial bifunctors \cite{FF}. Strict polynomial bifunctors are the two variables generalization of Friedlander and Suslin's strict polynomial functors \cite{FS}, and they provide a natural way of building representations of $GL_{n,\Bbbk}$. An important example of strict polynomial bifunctor is $\gl$, which is an object of the category $\PP^1_1$ of bifunctors of bidegree $(1,1)$; its underlying ordinary bifunctor is just the $\hom$-bifunctor on finite dimensional $\Bbbk$-vector spaces, and its value on $(\Bbbk^n,\Bbbk^n)$ is the adjoint representation of $GL_{n,\Bbbk}$. The cohomology $H^*_{\PP}(B)$ of a bifunctor $B$ is the stable (i.e. for $n$ big enough) cohomology of $GL_{n,\Bbbk}$ with coefficients in the representation provided by $B$. The core of the new proof is the following theorem (whose statement is explained in detail in section \ref{sec-cohom-twisted}).
\begin{theorem}\label{thm-transparent}
Let $B\in \PP^d_d$ be a strict polynomial bifunctor. There is a filtration of the bifunctor cohomology $H^*_{\PP}(B^{(r)})$, natural with respect to $B$, and a natural graded isomorphism 
$$H^*_\PP(B_{E_r})\simeq \Gr H^*_\PP(B^{(r)})\;, $$
where `$\Gr$' refers to the graded vector space associated with the filtration, $E_r$ is an explicit graded vector space, and $B_{E_r}$ is an object of $\mathcal{P}^d_d$ obtained as a graded modification of the bifunctor $B$ by a parametrization by $E_r$.
\end{theorem}
The interest of theorem \ref{thm-transparent} lies in the fact that the left hand side of the isomorphism is easier to compute. In the case $B=\Gamma^d\gl$, this isomorphism gives access to the bifunctor cohomology of $\Gamma^d(\gl^{(1)})$. In particular, we can prove the existence of the universal classes living in the cohomology of $GL_{n,\Bbbk}$ with coefficients in $\Gamma^d(\mathfrak{gl}_n^{(1)})$. 

As a common point, the two proofs rely on the very fundamental complexes constructed by Troesch in \cite{Troesch} (see also \cite[Section 9]{TouzeTroesch} for a slightly different presentation of these complexes).

\subsection*{Acknowledgements}

I thank Wilberd van der Kallen for carefully reading a first version of this paper and for pointing out 
a mistake hidden behind an ambiguous notation in the former proof of theorem \ref{thm-collapse}. 
I also thank Marcin Cha{\l}upnik for many discussions on the preprint \cite{Chal3}, which is the source of inspiration for the proof of
theorem \ref{thm-transparent}. 

\section{Functors and bifunctors}\label{sec-fun-bifun}

Our proof of theorem \ref{thm-univ} uses the category of strict polynomial bifunctors introduced in \cite{FF}. 
In this section, we recall the main facts that we will need about this category. The reader might consult \cite{FS,FF,Krause,TouzeKos} for more details
on strict polynomial functors and bifunctors. Throughout this section, $\Bbbk$ is a field of positive characteristic $p$.

\subsection{Functors}
Let us first begin with brief recollections of the simpler category of strict polynomial functors introduced by Friedlander and Suslin in \cite{FS}. 
We denote by $\PP_d$ the abelian category of homogeneous strict polynomial functors of degree $d$ over $\Bbbk$. 
The objects of $\PP_d$ can be thought of as `nice' endofunctors of the category $\V_\Bbbk$ of finite dimensional $\Bbbk$-vector spaces, 
which naturally arise in representation theory of algebraic groups, and the morphisms of $\PP_d$ are some natural transformations between these functors 
(see \cite[Def 2.1]{FS} for an explicit definition). 
Objects of $\PP_d$ include:
\begin{enumerate}
\item the $d$-th tensor power $\otimes^d:V\mapsto V^{\otimes d}$,
\item the $d$-th symmetric power $S^d:V\mapsto S^d(V)$,
\item the $d$-th divided power $V\mapsto \Gamma^d(V)=(V^{\otimes^d})^{\Si_d}$.
\end{enumerate}
The $r$-th Frobenius twist $I^{(r)}\in\PP_{p^r}$ is the subfunctor of $S^{p^r}$ such that $I^{(r)}(V)$ is generated by the elements of the form $v^{p^r}\in S^{p^r}(V)$, 
for $v\in V$.
If $F\in\PP_d$, we denote by $F^{(r)}$ the composition $F\circ I^{(r)}$. This is an object of $\PP_{dp^r}$.

If $F\in\PP_d$, we denote by  $F^\sharp$ its dual. By definition $F^\sharp(V):= F(V^\vee)^\vee$, where the symbol `$^\vee$' stands for $\Bbbk$-linear duality, 
and $F^\sharp$ is an object of $\PP_d$. We have a canonical isomorphism
$$\hom_{\PP_d}(F,G)\simeq\hom_{\PP_d}(G^\sharp,F^\sharp)\;. $$

If $F\in\PP_d$, and $X\in\V_\Bbbk$ we define parametrized variants of $F$ as follows:
$$F^X: V\mapsto F(\hom_\Bbbk(X,V))\;,\quad F_X: V\mapsto F(X\otimes V)\;.$$ 
The parametrized functors $F^X$ and $F_X$ are objects of $\PP_d$. The notation $F^X$ reminds that $F^{X}(V)$ is contravariant with respect to $X$ 
(compare with the usual notation for functional spaces), while $F_{X}(V)$ is covariant with respect to $X$. 
The usual adjunction on both sides of $X\otimes-$ and $\hom(X,-)=X^\vee\otimes-$ on finite-dimensional vector spaces induce that the two parametrization functors are adjoint on both sides, i.e. there are natural isomorphisms:
%These two parametrization functors are adjoint on both sides, i.e. we have natural isomorphisms:
$$\hom_{\PP_d}(F^X,G)\simeq \hom_{\PP_d}(F,G_X)\;,\qquad \hom_{\PP_d}(F_X,G)\simeq \hom_{\PP_d}(F,G^X)\;.$$

The family of functors $S^d_X:=(S^d)_X$, for $X\in\V_\Bbbk$ is an injective cogenerator of $\PP_d$ 
while the family of functors $\Gamma^{d,X}:=(\Gamma^d)^X= (S^d_X)^\sharp$, for $X\in\V_\Bbbk$, is a projective generator. We recall the isomorphisms, natural in $F,X$:
$$\hom_{\PP_d}(\Gamma^{d,X},F)\simeq F(X)\;,\quad \hom_{\PP_d}(F,S^d_X)\simeq F^\sharp(X)\;. $$
These isomorphisms are a form of the Yoneda lemma, see \cite{Krause,TouzeKos}, so we simply call them `the Yoneda isomorphisms'.

\begin{remark}\label{rk-Wilberd}
The notation $F_X$ for parametrized functors may come into conflict with the notation $F^{(r)}$ for precomposition by Frobenius twists. 
Indeed, Frobenius twists commute with tensor products, so we have $(F^{(r)})_X=(F_{X^{(r)}})^{(r)}$. Which is different from $(F_X)^{(r)}$. 
As pointed out by W. van der Kallen, we mustn't use the ambiguous notation $F_X^{(r)}$.
\end{remark}

\subsection{Bifunctors and bifunctor cohomology}
If $F\in\PP_d$, the vector space $F(\Bbbk^n)$ is canonically endowed with an action of the group scheme $GL_{n,\Bbbk}$, 
and Friedlander and Suslin proved \cite[Cor 3.13]{FS} that the evaluation map 
$$\Ext^*_{\PP_d}(F,G)\to \Ext^*_{GL_{n,\Bbbk}}(F(\Bbbk^n),G(\Bbbk^n))\simeq H^*\big(GL_{n,\Bbbk}, \hom_\Bbbk(F(\Bbbk^n),G(\Bbbk^n))\big)$$
is an isomorphism if $n\ge d$. (This allows to perform $\Ext$-computations in $\PP_d$, where computations are surprisingly easier).
Strict polynomial bifunctors were used in \cite{FF} to generalize this formula to more general $GL_{n,\Bbbk}$-representations 
than the ones of the somewhat restrictive form $\hom_\Bbbk(F(\Bbbk^n),G(\Bbbk^n))$. 
If $d,e\ge 0$, we denote by $\PP^d_e$ the abelian category of strict polynomial bifunctors 
which are homogeneous of bidegree $(d,e)$, contravariant with respect to their first variable and covariant with respect to their second variable.

Thus, objects of $\PP^d_e$ are bifunctors $B:(V,W)\mapsto B(V,W)$ taking a 
pair of finite dimensional vector spaces as arguments, with values in finite dimensional vector spaces. 
They satisfy the following condition: for all $W$ fixed, the functor $V\mapsto B(V,W)$ is a degree $d$ homogeneous 
\emph{contravariant} strict polynomial functor (or equivalently $V\mapsto B(V^\vee,W)$ is an object of $\PP_d$), and for all $V$ fixed,
the functor $W\mapsto B(V,W)$ is an object of $\PP_e$.
Examples of strict polynomial bifunctors can be constructed in the following ways.
\begin{enumerate}
\item If $F\in\PP_d$ and $G\in \PP_e$, we define a bifunctor $\hom_\Bbbk(F,G)\in \PP^d_e$ by the formula:
$$(V,W)\mapsto \hom_\Bbbk(F(V),G(W))\;.$$
\item If $F\in\PP_d$ we define a bifunctor $F\gl \in\PP^d_d$ by the formula: 
$$(V,W)\mapsto F(\hom_\Bbbk(V,W))\;.$$
\end{enumerate}

If $B$ is a strict polynomial bifunctor, the vector spaces $B(\Bbbk^n,\Bbbk^m)$ are canonically endowed with a left action of $GL_{m,\Bbbk}$
and a right action of $GL_{n,\Bbbk}$ which commute. By inverting matrices in $GL_{n,\Bbbk}$, we convert the right action into a left action. 
Taking furthermore $n=m$ and the diagonal action, we get an action of $GL_{n,\Bbbk}$ on $B(\Bbbk^n,\Bbbk^n)$. We have the following examples.
\begin{enumerate}
\item If $B=\hom_\Bbbk(F,G)$, the representation $\hom_\Bbbk(F(\Bbbk^n),G(\Bbbk^n))$ is the representation appearing as the cohomology 
coefficient in the target of Friedlander and Suslin's evaluation map above.
\item If $B=F\gl$, the representation $F\gl(\Bbbk^n,\Bbbk^n)$ is the representation $F(\mathfrak{gl}_n)$ obtained by evaluating $F$ 
on the adjoint representation $\mathfrak{gl}_n$ of $GL_{n,\Bbbk}$. %This representation does not appear as a cohomology 
%coefficient in the target of Friedlander and Suslin's evaluation map (except if $F=I^{(r)}$, because $I^{(r)}\gl=\hom_\Bbbk(I^{(r)},I^{(r)})$).
\end{enumerate}
Franjou and Friedlander proved \cite[Thm 1.5]{FF} the following generalization of Friedlander and Suslin's isomorphism. For all $n\ge 1$ there is a natural map 
$$\Ext^*_{\PP^d_d}(\Gamma^d\gl,B)\to H^*(GL_{n,\Bbbk},B(\Bbbk^n,\Bbbk^n))\;, $$
and this map is an isomorphism if $n\ge d$. For this reason, the extensions on the left hand side are called the `bifunctor cohomology of $B$' 
and written under the more suggestive notation $H^*_\PP(B)$.

The isomorphism above shows that when we compute the cohomology of bifunctors, we are dealing with finite dimensional objects: 
$H_\PP^i(B)$ is finite dimensional and it is zero if $i$ is big enough (indeed, this property holds for $H^*(GL_{n,\Bbbk}, B(\Bbbk^n,\Bbbk^n))$ 
by \cite[II 4.7 and II 4.10]{Jantzen})
\begin{remark}
If $B\in\PP^d_e$ with $d\ne e$, then the cohomology of $GL_{n,\Bbbk}$ with coefficients in $B(\Bbbk^n,\Bbbk^n)$ is zero in all degrees. 
Indeed, the center of $GL_{n,\Bbbk}$, i.e. the multiplicative group $\mathbb{G}_m$ of homotheties  acts non trivially (with weight $e-d$) on this representation, 
whereas the representations of $GL_{n,\Bbbk}$ providing nontrivial cohomology must be acted on trivially by the center of $GL_{n,\Bbbk}$. Hence, 
the bifunctors of bidegree $(d,d)$ are the only interesting bifunctors for us. They are also the only ones for which $H^*_\PP(B)$ is defined.
\end{remark}

\subsection{Some useful properties of bifunctors}
We now recall a few technical but useful definitions and formulas holding in strict polynomial bifunctor categories. We will need them in the sequel of the paper.

\begin{enumerate}
 \item If $B\in \PP^d_e$, and $r\ge 1$, we denote by $B^{(r)}$ the object of $\PP^{dp^r}_{ep^r}$ defined by precomposing $B$ by the $r$-th Frobenius twist on both variables:
$$B^{(r)}(V,W):= B(V^{(r)}, W^{(r)}) \;.$$
Bifunctors of this form are called \emph{twisted bifunctors}. 
\item If $B$ is a strict polynomial bifunctor, we denote by $B^{\sharp}$ its dual. To be more specific $B^\sharp(V,W)=B(V^\vee,W^\vee)^\vee$, where `$^\vee$' stands for $\Bbbk$-linear duality. Duality commutes with precomposition by Frobenius twists, i.e. there is a canonical isomorphism $(B^{(r)})^\sharp\simeq (B^\sharp)^{(r)}$. 
Moreover, it yields an equivalence of categories:
$$^\sharp:(\PP^d_e)^\op\xrightarrow[]{\simeq} \PP^d_e\;.$$
\item Bifunctors of the form $\hom_\Bbbk(F,G)$ with $F,G\in\PP_d$ are called \emph{bifunctors of separable type}. For this kind of bifunctors, computations in the category of strict polynomial bifunctors usually boil down to computations in the category $\PP_d$ via the following natural isomorphisms \cite[Prop 1.3 and 2.2]{FF}:
\begin{align*}
\quad\qquad\Ext^*_{\PP^d_e}(\hom_{\Bbbk}(F,G),\hom_{\Bbbk}(F',G'))&\simeq \Ext^*_{\PP_d}(F',F)\otimes \Ext^*_{\PP_e}(G,G')\;,\\
H^*_\PP(\hom_{\Bbbk}(F,G)) &\simeq \Ext^*_{\PP_d}(F,G)\;.
\end{align*}
\item The \emph{standard injectives} $J_{d,X,e,Y}$,  defined for $X\in\V_\Bbbk$ and $Y\in\V_\Bbbk$ by
$$J_{d,X,e,Y}=\hom_\Bbbk(\Gamma^d_X,S^e_Y)$$
form an injective cogenerator in the category $\PP^d_e$. Similarly, a projective generator of $\PP^d_e$ is provided by the family 
of the \emph{standard projectives} $P^{d,X,e,Y}$, defined for $X\in\V_\Bbbk$ and $Y\in\V_\Bbbk$ by:
$$P^{d,X,e,Y}=\hom_\Bbbk(S^d_X,\Gamma^e_Y)$$
There are natural isomorphisms:
$$\quad\qquad \hom_{\PP^d_e}(P^{d,X,e,Y},B)\simeq B(X,Y)\;,\quad \hom_{\PP^d_e}(B,J_{d,X,e,Y})\simeq B^\sharp(X,Y)\;. $$
These two isomorphisms are just forms of the Yoneda lemma, so we simply call them `the Yoneda isomorphisms'. Finally, the functor $S^d_{\Bbbk^d}$ and $\Gamma^d_{\Bbbk^d}$ contain the functor $\otimes^d$ as a direct summand. Hence, the bifunctors $P^{d,\Bbbk^d,d,\Bbbk^d}$ and $J^{d,\Bbbk^d,d,\Bbbk^d}$ both contain the bifunctor $\otimes^d\gl\simeq \hom_\Bbbk(\otimes^d,\otimes^d)$ as a direct summand. In particular, the latter is projective and injective.
\end{enumerate}

\section{The cohomology of twisted bifunctors} \label{sec-cohom-twisted}

In this section, we study the cohomology of twisted bifunctors. It is an interesting subject in its own, especially
in view of the main theorem of \cite{CPSVdK}. This theorem says that the cohomology of a (connected, defined over $\Fp$ and split over $\Fp$) 
reductive group $G$ with coefficients 
in a twisted representation $M^{(r)}$ with ($r$ big enough) is a good approximation of the ordinary cohomology of the \emph{finite group} of $\Fq$-points ($q$ big enough) $G(\Fq)$ 
with coefficients in $M\otimes\Fq$. In the case of $GL_{n,\Fp}$, this means that the bifunctor cohomology $H^*_\PP(B^{(r)})$, with $r$ big enough, 
can be used to compute the cohomology of the finite group $GL_n(\Fq)$ with coefficients in the representation $B(\Fq^n,\Fq^n)$.

Of course, the cohomology of twisted bifunctors is also interesting for the universal classes, since these classes live in 
the cohomology of $GL_{n,\Bbbk}$ with coefficients in the representation given by the twisted bifunctor $\Gamma(\gl^{(1)})=(\Gamma\gl)^{(1)}$.

Our approach to the cohomology of twisted bifunctors follows the ideas of \cite{TouzeTroesch} and \cite{Chal3}. The main result of the section is theorem 
\ref{thm-collapse} (it is equivalent to theorem \ref{thm-transparent} of the introduction), which gives an explicit formula to compute the cohomology of twisted bifunctors from the cohomology of untwisted bifunctors.

\subsection{The twisting spectral sequence}

We first generalize the twisting spectral sequence of \cite[Thm 7.1]{TouzeTroesch} to the framework of strict polynomial bifunctors. The second page of this
twisting spectral sequence 
is given by the bifunctor cohomology of some parametrized modification of $B$, and it converges to the cohomology of $B^{(r)}$. 

To state it clearly we need to introduce a notation. If $B\in\PP_{d}^d$ and $Z\in\V_\Bbbk$, we denote by $B_Z\in\PP_{d}^d$ the parametrized bifunctor:
$$B_Z:(V,W)\mapsto B(V,Z\otimes W)\;.$$
If moreover $Z$ is graded, then the functor $B_Z$ inherits a grading, defined similarly as in \cite[Section 2.5]{TouzeTroesch}. 
To be more specific, let the multiplicative group $\mathbb{G}_m$ act on each $Z^i$ with weight $i$ (that is $\lambda\cdot z:=\lambda^iz$), and trivially (i.e. with weight zero) on $V$ and $W$. 
Then $B(V,Z\otimes W)$ inherits an action of $\mathbb{G}_m$. By definition, the elements of $B(V,Z\otimes W)=B_Z(V,W)$ of degree $j$ are the elements of weight $j$ 
under this action of $\mathbb{G}_m$. 
It is not hard to check from this definition that if $Z$ is zero in negative degrees (resp. odd degrees), then $B_Z$ is zero in negative degrees (resp. odd degrees). In the sequel, we shall denote by $B_Z^t$ the summand of $B_Z$ of degree $t$: $B_Z=\bigoplus B_Z^t$.

In the twisting spectral sequence, we will need bifunctors parametrized by the graded vector space $E_r$ of dimension $p^r$ defined by:
$$ (E_r)^{i}=\Bbbk, \text{ if $i=2j$ with $0\le j<p^r$ } \text{ and $(E_r)^i=0$ otherwise}\;.$$
The graded vector space $E_r$ is isomorphic to $\Ext^*_{\PP_{p^r}}(I^{(r)},I^{(r)})$. This was first computed by Friedlander and Suslin \cite{FS}, see also \cite[Cor 4.8]{TouzeTroesch}.

\begin{example}
Let $X^d$ denote one of the functors $S^d$, $\Lambda^d$ or $\Gamma^d$. The parametrized bifunctor $(X^d\gl)_{E_r}$ is equal to the direct sum
$$(X^d\gl)_{E_r}=\bigoplus X^{d_0}\gl\otimes\dots\otimes X^{d_{p^r-1}}\gl$$
taken over all $p^r$-tuples of nonnegative integers $(d_0,\dots,d_{p^r-1})$ satisfying $\sum d_i=d$. 
The summand  $X^{d_1}\gl\otimes\dots\otimes X^{d_{p^r}}\gl$ has degree $\sum 2i d_i$. 
\end{example}

The following twisting spectral sequence is a special case of hypercohomology spectral sequence, as all the spectral sequences that we will use in this paper. We use the conventions of \cite[Chap. 5]{Weibel} for our spectral sequences. Namely, the differentials of the $r$-th page are maps $d_r^{s,t}:E_r^{s,t}\to E_r^{s+r,t-r+1}$. The convergence $E_2^{s,t}\Rightarrow H^{s+t}$ means that there is a finite filtration:
$$ 0= F^{s+t+1}H^{s+t}\subset F^{s+t}H^{s+t}\subset \dots \subset F^1H^{s+t} \subset F^0H^{s+t}=H^{s+t}$$ 
such that $E_\infty^{s,t}$ is isomorphic to the graded piece $\Gr^s H^{s+t}=F^s H^{s+t}/ F^{s+1} H^{s+t}$. 
\begin{proposition}[The twisting spectral sequence]\label{prop-twss} Let $B\in\PP^d_d$. Then the parametrized functor $B_{E_r}$ is graded, and we denote its grading by the letter `$t$'.
There is a first quadrant cohomological spectral sequence, natural with respect to $B$: 
$$ E_2^{s,t}(B,r)= H^s_\PP(B_{E_r}^t)\Longrightarrow H^*_\PP(B^{(r)})\;,$$
\end{proposition}
\begin{proof}
Let $P$ be a projective resolution of $\Gamma^{p^rd}\gl$, and let $K$ be an injective resolution of $B$. There are two spectral sequences associated to the bicomplex $\hom_{\PP^{p^rd}_{p^rd}}(P,K^{(r)})$, and converging to the homology of the total complex. The first one is  
$$E_2^{s,t}(B,r)= H^s(H^t_\PP(K^{(r)})) \Longrightarrow H^{s+t}(\hom_{\PP^{p^rd}_{p^rd}}(P,K^{(r)}))$$
This will be our twisting spectral sequence. To identify the abutment, we use the second spectral sequence associated to the bicomplex. Its second page is concentrated on one row, hence it collapses to a natural isomorphism:
$$H^{*}(\hom_{\PP^{p^rd}_{p^rd}}(P,K^{(r)}))\simeq H^{*}_\PP(B^{(r)})\;.$$
To identify the second page, we use the natural isomorphism \cite[Thm 4.6]{TouzeTroesch} or \cite[Thm 4.3]{Chal1} for all projective $P$ and all injective $J$:
$$\Ext^*(P^{(r)},J^{(r)})\simeq \hom(P,J_{E_r}^*)\;,$$
which can be rewritten in terms of bifunctors as
$$H^*_\PP(\hom_\Bbbk(P,J)^{(r)})\simeq H^0_\PP(\hom_\Bbbk(P,J)_{E_r}^*)\;.$$
Since the bifunctors of the form $\hom_\Bbbk(P,J)$ are the standard injective bifunctors, this means that we have an isomorphism of complexes:
$$H^t_\PP(K^{(r)})\simeq H^0_\PP(K_{E_r}^t)\;. $$ 
Now parametrization by $E_r$ (or by any graded vector space) is exact and preserves injective objects. 
Thus $K_{E_r}$ is an injective resolution of the graded bifunctor $B_{E_r}$, 
so the degree $s$ cohomology of the complex $H^t_\PP(K^{(r)})$ equals $H^s_\PP(B_{E_r}^t)$ as asserted.\qed
\end{proof}

The convergence of the twisting spectral sequence means that $E_\infty^{s,t}(B,r)$ is naturally isomorphic to $\Gr^s H^{s+t}_\PP(B^{(r)})$, where `$\Gr$' 
refers to a filtration of $H^*_\PP(B^{(r)})$ which is natural with respect to $B$. Since the graded objects are finite dimensional vector spaces, this implies that
 $E_\infty^{*,*}(B,r)$ (with total degree) 
is isomorphic (in a unnatural way) to $H^*_\PP(B^{(r)})$. 

The goal of section \ref{sec-cohom-twisted} is to prove 
that the twisting spectral sequence stops at the second page, i.e. that $E_2^{s,t}(B,r)$ equals $E_\infty^{s,t}(B,r)$.
In the framework of strict polynomial functors, a similar statement was conjectured and proved in many cases in \cite{TouzeTroesch}. 
The general case was announced by Cha{\l}upnik in \cite{Chal3}. 

\subsection{The collapsing theorem}\label{subsec-collapse}

We now state the main result of section \ref{sec-cohom-twisted}. It is equivalent to theorem \ref{thm-transparent} in the introduction.

\begin{theorem}\label{thm-collapse}
For all $B\in\PP^d_d$ and all $r\ge 1$, the twisting spectral sequence stops at the second page, i.e. $E_2^{s,t}(B,r)=E_\infty^{s,t}(B,r)$. 
Equivalently, there is an (a priori not natural) isomorphism of graded vector spaces (take the total degree on the left hand side)
$$ H^*_\PP(B_{E_r})\simeq H^*_\PP(B^{(r)})\;.$$
\end{theorem}

This theorem generalizes the case of strict polynomial functors announced in \cite{Chal3}. Indeed, the case of bifunctors of the form $\hom_\Bbbk(F,G)$ implies the following result, which proves \cite[Conjecture 8.1]{TouzeTroesch}.
\begin{corollary}
Let $F,G\in\PP_d$. Then $G_{E_r}$ is a graded functor and there is an (a priori not natural) isomorphism of graded vector spaces
 (take the total degree on the left hand side)
$$\Ext^*_{\PP_d}(F,G_{E_r})\simeq \Ext^*_{\PP_{dp^r}}(F^{(r)},G^{(r)})\;.$$
\end{corollary}

The remainder of section \ref{sec-cohom-twisted} is devoted to the proof of theorem \ref{thm-collapse}. 
We first observe that it suffices to prove theorem \ref{thm-collapse} for the case $r=1$. Indeed we have the following analogue of \cite[Remark 8.2]{TouzeTroesch}.
\begin{lemma}\label{lm-red}
Assume that for all $d\ge 1$ and all $B\in\PP^d_d$ the twisting spectral sequence $(E^{*,*}_k(B,1))_{k\ge 2}$ stops at page $2$. 
Then for all $r\ge 2$ the twisting spectral sequence $(E^{*,*}_k(B,r))_{k\ge 2}$ also stops at page $2$. 
\end{lemma}
\begin{proof}
Since the graded vector spaces in play are finite dimensional, $(E^{*,*}_k(B,r))_{k\ge 2}$ stops at the second page if and only if 
$$\dim H^*_\PP(B^{(r)})=\dim H^*_\PP(B_{E_r})\;,$$
where $\dim$ refers to the total dimension, i.e. the dimension of the underlying ungraded objects.
Assume that we know such an equality for all $B$ and for a given $r$. Then the same equality holds for $r+1$. Indeed:
$$\dim H^*_\PP(B^{(r+1)})= \dim H^*_\PP((B^{(1)})^{(r)}) = \dim H^*_\PP((B^{(1)})_{E_r})\;.$$
But $(B^{(1)})_{E_r}$ is isomorphic to $(B_{E_r})^{(1)}$ in an ungraded way, because $(E_r)^{(1)}$ and $E_r$ have the same dimension. Hence
$$\dim H^*_\PP((B^{(1)})_{E_r})=\dim H^*_\PP((B_{E_r})^{(1)})=\dim H^*_\PP((B_{E_r})_{E_1})\;.$$
Now $\dim E_r\otimes E_1$ equals $\dim E_{r+1}$, hence $(B_{E_r})_{E_1}$ is isomorphic to $B_{E_{r+1}}$ in an ungraded way. So we finally obtain that
$\dim H^*_\PP(B^{(r+1)})$ equals $\dim H^*_\PP(B_{E_{r+1}})$. 
The proof of lemma \ref{lm-red} follows by induction on $r$. \qed
\end{proof}

So we concentrate on the case $r=1$ of theorem \ref{thm-collapse}. The plan of the proof is as follows. 

{\bf Step 1 (section \ref{subsec-adj}).} We show that the twisting spectral sequence is isomorphic to the first hypercohomology spectral sequence $'E^{*,*}_k$ computing 
the homology of the complex
$\RR \hom_{\PP_d^d}(\LL\ell(\Gamma^{pd}gl),B)$. 
For this, we use Cha{\l}upnik's adjunction argument \cite{Chal3} suitably adapted to the framework of bifunctors: the letter `$\ell$' appearing inside the $\RR\hom$ 
refers to the left adjoint of the precomposition by the Frobenius twist functor $I^{(1)}$. 

{\bf Step 2 (section \ref{subsec-twist-collapse}).} Precomposition by the Frobenius twist $I^{(1)}$ induces a 
morphism of spectral sequences
$\phi_k:{'E^{*,*}_k}\to {''E^{*,*}_k}$
where ${''E^{*,*}_k}$ is the first hypercohomology spectral sequence computing 
the homology of the complex
$\RR \hom_{\PP_{pd}^{pd}}(\LL\ell(\Gamma^{pd}gl)^{(1)},B^{(1)})$.
We prove that this morphism is an injection on the second page, using the Frobenius twist injectivity \cite[II.10.16]{Jantzen}. 
Thus, to prove that $'E^{*,*}_k$ stops at the second page, 
it suffices to prove that  $''E^{*,*}_k$ stops at the second page.

{\bf Step 3 (section \ref{subsec-formal}).} We prove that the complex $\LL\ell(\Gamma^{pd}gl)^{(1)}$ is formal, 
hence the hypercohomology spectral sequence ${''E^{*,*}_k}$ stops at the second page, and this finishes the proof of theorem \ref{thm-collapse}.

\begin{remark}
In the plan of the proof above as well as in the proof itself, we use standard concepts of derived categories like $\RR \hom$, $\LL\ell$. We do this to provide a more conceptual description of the complexes appearing in the proof. However, our proof actually holds at the level of chain complexes: we construct explicit chain complexes, and explicit zig-zags of chain maps between them. The reader can consult \cite[Chap. 10]{Weibel} for an introduction to derived categories.
%However, we do the proof at the level of chain complexes, we never use the properties of derived categories. Thus, the reader uncomfortable with derived categories can think of these derived category concepts as names for the various complexes appearing in the proof: $\LL\ell(\Gamma^{pd}gl)$ denotes the complex $\ell(P)$ where $P$ is the projective resolution of $\Gamma^{pd}gl$ chosen in the proof of proposition \ref{prop-st1}, and this definition is valid throughout section \ref{sec-cohom-twisted} (i.e. we never change the projective resolution $P$), $\RR \hom_{\PP_d^d}(C,B)$ is the total complex of the bicomplex $\hom_{\PP_d^d}(C,K)$ where $K$ is the injective resolution of $B$ chosen in the proof of proposition \ref{prop-st1}, and so on.  
\end{remark}

\subsection{Step1 : the adjunction argument}\label{subsec-adj}

In this paragraph, we adapt the adjunction argument of \cite{Chal3} to the category of strict polynomial bifunctors.
If $B\in\PP^d_d$, precomposition by the Frobenius twist functor $I^{(1)}$ on both variables induces an exact functor:
$$\begin{array}{cccc}
\Tw:&\PP^d_d&\to  &\PP^{pd}_{pd}\\
& B &\to & B^{(1)}   
  \end{array}
\;.$$
If $B\in \PP^{dp}_{dp}$, we introduce a bifunctor $\ell(B)\in \PP^d_d$. To be more specific, we define its \emph{dual} $\ell(B)^\sharp$ as the bifunctor:
$$(X,Y)\mapsto \hom_{\PP^{pd}_{pd}}(B,(J_{d,X,d,Y})^{(1)})\;. $$
Here $J_{d,X,d,Y}$ is the standard injective.
This yields a functor 
$$\ell:\PP^{dp}_{dp}\to \PP^d_d\;.$$

\begin{proposition}\label{prop-adjoint}
The functors $(\ell, \Tw)$ form an adjoint pair. 
\end{proposition}
\begin{proof}
It suffices to prove an isomorphism, natural in $B$, $B'$: 
$$\hom_{\PP^d_d}(\ell(B),B')\simeq\hom_{\PP^{pd}_{pd}}(B,(B')^{(1)}) \;,\quad (*)$$
when  $B=P^{pd,V,pd,W}$ and $B'= J_{d,X,d,Y}$ (the case of general $B$, $B'$ follows by taking resolutions). In this case,
the left hand side of formula $(*)$ is isomorphic to $\ell(B)^\sharp(X,Y)$ via the Yoneda isomorphism, and the latter equals the right hand side of formula $(*)$ by definition of $\ell(B)$.\qed
\end{proof}

We use the notations of \cite{Weibel} for categories of complexes and derived categories. Thus, $\Ch^+(\PP^{pd}_{pd})$ denotes the category of bounded below cochain complexes of bifunctors, i.e. cochain complexes $C$ with $C^i=0$ for $i\ll 0$. Similarly $\Ch^-(\PP^{pd}_{pd})$, denotes the category of bounded above cochain complexes of bifunctors. The corresponding derived (i.e. localized at quasi-isomorphisms) categories are respectively denoted by $\DD^+(\PP^{pd}_{pd})$ and $\DD^-(\PP^{pd}_{pd})$. The functor $\ell$ induces a functor 
$$\LL\ell: \DD^-(\PP^{pd}_{pd})\to \DD^-(\PP^{d}_{d})\;,$$
whose value on a bifunctor $B$ (considered as a complex concentrated in degree zero) is quasi-isomorphic to the complex $\ell(P)$, where $P$ is a projective resolution of $B$. Similarly, if $C$ is a bounded above cochain complex, there is a functor 
$$\RR\hom_{\PP^{d}_{d}}(C,-): \DD^+(\PP^{d}_{d})\to \DD^+(\PP^{d}_{d})\;,$$
whose value on a bifunctor $B$ is quasi-isomorphic to the (total) complex $\hom_{\PP^{d}_{d}}(C,K)$, where $K$ is an injective resolution of $B$.

Let $C$ be a complex of bifunctors with $C_i=0$ if $i<0$, let $B\in\PP^d_d$ be a bifunctor, and let $K$ be an injective resolution of $B$. 
There are two first quadrant cohomological
spectral sequences associated to the bicomplex $\hom_{\PP^d_d}(C,K)$, and converging to the homology of the total complex, that is, to the homology of $\RR\hom_{\PP^d_d}(C,B)$:  
\begin{align*}
{^{I}}E_2^{s,t}= \Ext^s_{\PP^d_d}(H_tC,B)&\Longrightarrow H^{s+t}\RR\hom_{\PP^d_d}(C,B)\;,\\
{^{II}}E_2^{s,t}= H^s(\Ext^t_{\PP^d_d}(C,B))&\Longrightarrow H^{s+t}\RR\hom_{\PP^d_d}(C,B)\;.
\end{align*}

\begin{definition}
We shall refer to the the spectral sequence $^IE^{*,*}_k$ as \emph{the first hyperhomology spectral sequence computing the homology of $\RR\hom_{\PP^d_d}(C,B)$}.  
\end{definition}

The first step of the proof of theorem \ref{thm-collapse} is the following statement.

\begin{proposition}[Step 1]\label{prop-st1}
Let $B\in\PP_d^d$.
The twisting spectral sequence $E^{*,*}_k(B,1)$ is isomorphic to the 
first hypercohomology spectral sequence computing the homology of $\RR\hom_{\PP^d_d}(\LL\ell(\Gamma^{pd}\gl),B)$.
\end{proposition}
\begin{proof}
Let $K^*$ be an injective resolution of $B$, and let $P_*$ be a projective resolution of $\Gamma^{pd}\gl$.
By definition (see the proof of proposition \ref{prop-twss}) the twisting spectral sequence is the first spectral sequence associated to the bicomplex 
$\hom_{\PP^{pd}_{pd}}(P,K^{(1)})$.
The latter is isomorphic to the bicomplex $\hom_{\PP^d_d}(\ell(P),K)$, which, by definition of the total derived functor $\LL\ell$, can be written as 
$\hom_{\PP^d_d}(\LL\ell(\Gamma^{pd}\gl),K^*)$. This concludes the proof. \qed
\end{proof}

\subsection{Step 2: twist injectivity}\label{subsec-twist-collapse}
Precomposition by the Frobenius twist $I^{(1)}$ is exact, so it induces a morphism (natural with respect to $B$, $B'$):
$$\Tw^*:\Ext^*_{\PP^d_d}(B,B')\to \Ext^*_{\PP^{dp}_{dp}}(B^{(1)},{B'}^{(1)})\;.$$
Precomposition by the Frobenius twist is the analogue of the following construction for $GL_{n,\Bbbk}\times GL_{m,\Bbbk}$-modules \cite[I.9.10]{Jantzen}. If $M$ is a $GL_{n,\Bbbk}\times GL_{m,\Bbbk}$-module, we denote by $M^{[1]}$ the $\Bbbk$-module $M$, equipped with the action of $GL_{n,\Bbbk}\times GL_{m,\Bbbk}$ where a pair of matrices $([a_{i,j}],[b_{k,\ell}])$ acts in $M^{[1]}$ as the pair $([a_{i,j}^{p}], [b_{k,\ell}^{p}])$ acts in $M$. The following result is straightforward from the definitions.
\begin{lemma}\label{lm-1}
There is a commutative diagram, where the horizontal arrows are induced by evaluating bifunctors on the pair $(\Bbbk^n,\Bbbk^m)$:
$$\xymatrix{
\Ext^*_{\PP^d_d}(B,B')\ar[d]\ar[r]& \Ext^*_{GL_{n,\Bbbk}\times GL_{m,\Bbbk}}(B(\Bbbk^n,\Bbbk^m),B'(\Bbbk^n,\Bbbk^m))\ar[d]\\
\Ext^*_{\PP^{dp}_{dp}}(B^{(1)},{B'}^{(1)})\ar[r] & \Ext^*_{GL_{n,\Bbbk}\times GL_{m,\Bbbk}}(B(\Bbbk^n,\Bbbk^m)^{[1]},B'(\Bbbk^n,\Bbbk^m)^{[1]})
}\;.$$ 
\end{lemma}

The following lemma is a slight generalization of \cite[Cor 3.13]{FS}. It can be also obtained as a corollary of \cite[Thm 4.5]{TouzeClassical}.

\begin{lemma}\label{lm-2}
Let $B,B'\in\PP^{d}_d$. If $n\ge d$ and $m\ge d$, the evaluation functor induces an isomorphism:
$$\Ext^*_{\PP^d_d}(B,B')\simeq \Ext^*_{GL_{n,\Bbbk}\times GL_{m,\Bbbk}}(B(\Bbbk^n,\Bbbk^m),B'(\Bbbk^n,\Bbbk^m))\;.$$ 
\end{lemma}
\begin{proof}
This holds true when $B$ and $B'$ are of separable type, i.e. of the form $\hom_\Bbbk(F,G)$ and $\hom_\Bbbk(F',G')$. Indeed, we have a commutative diagram, where the upper horizontal arrow is an isomorphism by Friedlander and Suslin's result \cite[Cor 3.13]{FS}.
$$ \text{\small $
\xymatrix{
\Ext^*_{\PP_d}(F^\sharp,{F'}^\sharp)\otimes \Ext^*_{P_d}(G,G')\ar[r]\ar[d]^{\otimes}_\simeq& \Ext^*_{GL_n}(F^\sharp(\Bbbk^n),{F'}^\sharp(\Bbbk^n))\otimes \Ext^*_{GL_m}(G(\Bbbk^n),G'(\Bbbk^n))\ar[d]^{\otimes}_\simeq\\
\Ext^*_{\PP_d^d}(\hom_\Bbbk(F,G),\hom_\Bbbk(F',G'))\ar[r]& \Ext^*_{GL_n\times GL_m}(\hom_\Bbbk(F(\Bbbk^n),G(\Bbbk^m)),\hom_\Bbbk(F'(\Bbbk^n),G'(\Bbbk^m)))
&&}
$}
$$
In particular, the lemma is valid when $B$ is a standard projective and $B'$ is a standard injective. The general result now follows by a standard $\delta$-functor argument. \qed
\end{proof}

We can now state the twist injectivity for bifunctors.

\begin{proposition}\label{prop-tw-inj}
Let $B,B'\in\PP^d_d$. Precomposition by $I^{(1)}$ yields an injective map:
$$\Ext^*_{\PP^d_d}(B,B')\hookrightarrow \Ext^*_{\PP^{dp}_{dp}}(B^{(1)},{B'}^{(1)})\;.$$
\end{proposition}
\begin{proof}
In view of lemmas \ref{lm-1} and \ref{lm-2}, it suffices to prove the analogous statement for the corresponding $GL_{n,\Bbbk}\times GL_{m,\Bbbk}$-modules. This statement is a theorem of Cline Parshall and Scott \cite{CPS}, see also \cite[II.10.16]{Jantzen}. \qed
\end{proof}

Let $K$ denote an injective resolution of $B'$, and $L$ an injective resolution of ${B'}^{(1)}$. The identity lifts to a morphism of complexes $K^{(1)}\to L$. The morphism $\Tw^*$ can be described at the cochain level as the composite
$$\hom_{\PP^d_d}(B,K)\to \hom_{\PP^{pd}_{pd}}(B^{(1)},K^{(1)})\to \hom_{\PP^{pd}_{pd}}(B^{(1)},L)\;.$$
In particular, by replacing $B$ by the complex $\LL\ell(\Gamma^{dp}\gl)$, we get a morphism of bicomplexes
$$\hom_{\PP^d_d}(\LL\ell(\Gamma^{dp}\gl),K)\to \hom_{\PP^d_d}(\LL\ell(\Gamma^{dp}\gl)^{(1)},L)\;.$$
This morphism of bicomplexes induces a morphism between the associated hypercohomology spectral sequences, which is nothing but the morphism
$$\Tw^s: \Ext^s(\,H_t(\LL\ell(\Gamma^{pd}\gl))\,,B)\to \Ext^s(\,H_t(\LL\ell(\Gamma^{pd}\gl)^{(1)})\,, B^{(1)})\;.$$
Hence, from proposition \ref{prop-tw-inj}, we get the following statement, which completes the second step of the proof of theorem \ref{thm-collapse}.
\begin{proposition}[Step 2]\label{prop-st2}
Let us denote by $'E_k^{*,*}$, resp. $''E_k^{*,*}$, the first hypercohomology spectral sequence computing $\RR\hom(\LL\ell(\Gamma^{pd}\gl),B)$, resp. $\RR\hom(\LL\ell(\Gamma^{pd}\gl)^{(1)},B^{(1)})$.
Precomposition by the Frobenius twist induces a morphism of spectral sequences:
$$'E_k^{*,*}\to ''E_k^{*,*}\;.$$
Moreover, this morphism is injective at page two.
\end{proposition}

\subsection{Step 3: formality}\label{subsec-formal}
The following formality statement is a consequence of the key lemmas \cite[Lemma 4.4, Lemma 4.5]{TouzeTroesch}.
\begin{proposition}\label{prop-formal}
The complex $\LL\ell(\Gamma^{dp}\gl)^{(1)}$ is formal: there is an isomorphism in the derived category $\DD^-(\PP^{pd}_{pd})$ (the object on the right hand side is concentrated in even degrees, with zero differential):
$$\LL\ell(\Gamma^{dp}\gl)^{(1)}\simeq ((\Gamma^{d}\gl)_{E_r})^{(1)}$$
\end{proposition}
\begin{proof}
It suffices to prove the dual statement, i.e. that the complexes
$$(\LL\ell(\Gamma^{dp}\gl)^{(1)})^\sharp = (\LL\ell(\Gamma^{dp}\gl)^\sharp)^{(1)}\;,\text{ and } ((S^{d}\gl)_{E_r})^{(1)}$$
are isomorphic in $\DD^+(\PP^{pd}_{pd})$. To prove this, we are going to write an explicit zigzag in $\Ch^+(\PP^{pd}_{pd})$ between these two complexes (the chain maps in this zigzag will be quasi-isomorphisms). We will denote by $(X,Y)$ the variables of the bifunctors in play.  

{\bf Step 1.} Let $P$ be a projective resolution of $\Gamma^{pd}\gl$. By definition, we have
$$\LL\ell(\Gamma^{dp}\gl)^{\sharp}(X^{(1)},Y^{(1)})=\hom_{\PP^{dp}_{dp}}(P, (J_{d,X^{(1)},d, Y^{(1)}})^{(1)})\;.$$
There is an equality (as already observed in remark \ref{rk-Wilberd}, the twist on $X$ and $Y$ disappears on the right hand side)
$$(J_{d,X^{(1)},d, Y^{(1)}})^{(1)}=\hom_\Bbbk((\Gamma^{d(1)})_X, (S^{d(1)})_Y)\;.$$ 
Let $T$, resp. $Q$ denote an injective resolution of $S^{d(1)}$, resp. a projective resolution of $\Gamma^{d(1)}$. Then $T_Y$ is an injective resolution of $(S^{d(1)})_Y$, and $Q_X$ is a projective resolution of $(\Gamma^{d(1)})_X$. The morphisms $P_0\twoheadrightarrow \Gamma^{pd}\gl$, $S^{d(1)}\hookrightarrow T^0$ and $Q_0\twoheadrightarrow \Gamma^{d(1)}$ induce chain maps (take total complexes everywhere):
\begin{align*}
&\hom_{\PP^{pd}_{pd}}(P, \hom_\Bbbk((\Gamma^{d(1)})_X, (S^{d(1)})_Y))\hookrightarrow \hom_{\PP^{pd}_{pd}}(P, \hom_\Bbbk(Q_X, T_Y))\;,\\
&\hom_{\PP^{pd}_{pd}}(\Gamma^{dp}\gl, \hom_\Bbbk(Q_X, T_Y)) \hookrightarrow \hom_{\PP^{pd}_{pd}}(P, \hom_\Bbbk(Q_X, T_Y))\;.
\end{align*}
These two chain maps are quasi-isomorphisms. 

{\bf Step 2.} Since the complex $\hom_\Bbbk(Q_X, T_Y)$ involves bifunctors of separable type, we have an isomorphism of complexes:
$$\hom_{\PP^{pd}_{pd}}(\Gamma^{dp}\gl, \hom_\Bbbk(Q_X, T_Y)) \simeq \hom_{\PP_{dp}}(Q_X, T_Y)\;.$$
Moreover, the morphism $Q_0\twoheadrightarrow \Gamma^{d(1)}$ induces a chain map:
$$\hom_{\PP_{dp}}((\Gamma^{d(1)})_X, T_Y) \to \hom_{\PP_{dp}}(Q_X, T_Y)\;,$$
which is a quasi-isomorphism. Finally, adjointness of upper and lower parametrizations yields an isomorphism of complexes
$$\hom_{\PP_{dp}}((\Gamma^{d(1)})_X, T_Y)\simeq \hom_{\PP_{dp}}((\Gamma^{d(1)})_X^Y, T)\;.$$ 
Observe that, unlike the notation $F^{(r)}_X$ alluded to in remark \ref{rk-Wilberd}, the notation $F_X^Y$ is unambiguous since the functors $X\otimes-$ and $\hom(Y,-)$ commute up to a canonical isomorphism. Hence we have canonical isomorphisms $(F_X)^Y\simeq (F^Y)_X\simeq F^{\hom_\Bbbk(X,Y)}$.

{\bf Step 3.} Now we use the key lemmas of \cite{TouzeTroesch}. We can assume that the injective resolution $T$ is the Troesch resolution of $S^{d(1)}$. 
We have an equality:
$$(\Gamma^{d(1)})_X^Y=((\Gamma^d)_{X^{(1)}}^{Y^{(1)}})^{(1)}=((\Gamma^{d})^{\hom_\Bbbk(X^{(1)},Y^{(1)})})^{(1)}\;.$$ 
By \cite[Lemma 4.4]{TouzeTroesch}, the complex  
$$\hom_{\PP_{dp}}(((\Gamma^{d})^{\hom_\Bbbk(X^{(1)},Y^{(1)})})^{(1)}, T)$$
is zero in even degrees (hence it is formal), and by \cite[Lemma 4.5]{TouzeTroesch} there is an isomorphism:
$$\hom_{\PP_{dp}}(((\Gamma^{d})^{\hom_\Bbbk(X^{(1)},Y^{(1)})})^{(1)}, T)\simeq \hom_{\PP_{dp}}((\Gamma^{d})^{\hom_\Bbbk(X^{(1)},Y^{(1)})}, S^d_{E_r})\;.$$
By the Yoneda lemma, the latter is isomorphic to $S^d(E_r\otimes \hom_\Bbbk(X^{(1)},Y^{(1)}))$. This concludes the proof. \qed
\end{proof}

As a consequence of proposition \ref{prop-formal}, we obtain the following result, which concludes the third step of the proof of theorem \ref{thm-collapse}.
\begin{proposition}[Step 3]\label{prop-st3}
The first hypercohomology spectral sequence computing the homology of $\RR\hom_{\PP^{pd}_{pd}}(\LL\ell(\Gamma^{pd}\gl)^{(1)},B^{(1)})$ stops at the second page. 
\end{proposition}
\begin{proof} Proposition \ref{prop-formal} implies that
the two complexes 
$$\RR\hom_{\PP^{pd}_{pd}}(\LL\ell(\Gamma^{pd}\gl)^{(1)},B^{(1)})\;,\text{ and } \RR\hom_{\PP^{pd}_{pd}}((\Gamma^{pd}\gl_{E_r})^{(1)},B^{(1)})$$ 
are isomorphic in $\DD^+(\PP^{pd}_{pd})$ and they have isomorphic hypercohomology spectral sequences. The spectral sequence associated to the second one stops at the second page: it is the spectral sequence associated to a bicomplex whose first differential is zero. Whence the result. \qed
\end{proof}

As indicated in the plan of the proof in section \ref{subsec-collapse}, propositions \ref{prop-st1}, \ref{prop-st2} and \ref{prop-st3} together prove that theorem \ref{thm-collapse} holds for $r=1$. Hence by lemma \ref{lm-red}, it holds for any $r\ge 1$.

\section{Proof of theorem \ref{thm-univ}}\label{sec-proof}

In this section, we prove theorem \ref{thm-univ}, i.e. we build classes in the cohomology of $GL_{n,\Bbbk}$ with coefficients in $\Gamma^d(\mathfrak{gl}_n^{(1)})$, where $\mathfrak{gl}_n^{(1)}$ is the representation obtained by base change along the Frobenius morphism. We first transpose the problem in the framework of strict polynomial functors as in \cite[Section 1.2]{TouzeUniv}. Since the representation $\mathfrak{gl}_n$ is defined over $\Fp$, we have an isomorphism $\mathfrak{gl}_n^{(1)}\simeq \mathfrak{gl}_n^{[1]}$ \cite[I.9.10]{Jantzen}. In particular, the representation $\Gamma^d(\mathfrak{gl}_n^{(1)})$ is obtained by evaluating the bifunctor $(\Gamma^d\gl)^{(1)}$ on the pair $(\Bbbk^n,\Bbbk^n)$. Using the evaluation map 
$$H^*_{\PP}(B)\to H^*(GL_{n,\Bbbk},B(\Bbbk^n,\Bbbk^n))\,, $$
together with the fact that for $n\ge p$, it induces an isomorphism
$$H^2_{\PP}(\gl^{(1)})\xrightarrow[]{\simeq}H^2(GL_{n,\Bbbk},\mathfrak{gl}_n^{(1)})\;,$$
we easily see that theorem \ref{thm-univ} is implied by the following theorem. (The cup product on bifunctor cohomology appearing in theorem \ref{thm-bis} is defined in \cite[Section 1]{TouzeUniv} by taking tensor products of extensions and pulling back by the diagonal $\Gamma^{d+e}\gl\to (\Gamma^d\gl)\otimes (\Gamma^e\gl)$.)
\begin{theorem}\label{thm-bis}
Let $\Bbbk$ be a field of positive characteristic $p$. There are cohomology classes $c[d]\in H^{2d}_\PP((\Gamma^{d}\gl)^{(1)})$ satisfying the following conditions.
\begin{enumerate}
\item $c[1]$ is non zero.
\item If $d\ge 1$ and $\Delta_{(1,\dots,1)} : (\Gamma^{d}\gl)^{(1)}\to (\otimes^{d}\gl)^{(1)}$ is the inclusion, then $\Delta_{(1,\dots,1)\,*}c[d]=c[1]^{\cup d}$.
\end{enumerate}
\end{theorem}

So we are left with the problem of finding classes $c[d]\in H^{2d}_\PP((\Gamma^{d}\gl)^{(1)})$. Finding $c[1]$ is not a problem. 
Indeed, it is well-known that $H^{2}_\PP(\gl^{(1)})$ is one dimensional (this is equivalent to the fact that $\Ext^2_{\PP_p}(I^{(1)},I^{(1)})$ is one dimensional, which is known since the seminal article \cite{FS}). So we choose for $c[1]$ a non zero cohomology class in $H^{2}_\PP(\gl^{(1)})$.

Now we want to find the classes $c[d]$ for $d\ge 2$. The action of the symmetric group $\Si_d$ on $\otimes^d$ (by permuting the factors of the tensor product) induce an action on the graded vector space $H^*_\PP((\otimes^{d}\gl)^{(1)})$.
\begin{lemma}\label{lm-invariant}
For all $d\ge 2$ the cup product $c[1]^{\cup d}\in H^{2d}_\PP((\otimes^{d}\gl)^{(1)})$ is invariant under the action of the symmetric group.
\end{lemma}
\begin{proof}
Recall that the cup product:
$$H^i_\PP(B)\otimes H^j_\PP(B')\xrightarrow[]{\cup}H^{i+j}(B\otimes B')$$
is defined as the composite of the external product of extensions
$$\Ext^i_{\PP^d_d}(\Gamma^d\gl,B)\otimes \Ext^j_{\PP^e_e}(\Gamma^d\gl,B)\to \Ext^{i+j}_{\PP^{d+e}_{d+e}}(\Gamma^d\gl\otimes\Gamma^e\gl, B\otimes B')$$
and the map induced by the comultiplication $\Gamma^{d+e}\gl\to \Gamma^{d}\gl\otimes\Gamma^e\gl$.
Since $c[1]$ is in even degree and $\Gamma^{*}\gl$ is a cocommutative coalgebra, one easily gets the result from the definition of the action of $\Si_d$ and the definition of the cup product. \qed
\end{proof}

In view of lemma \ref{lm-invariant}, to get a proof of theorem \ref{thm-bis}, it suffices to prove that all the classes of $H^{2d}_\PP((\otimes^{d}\gl)^{(1)})^{\Si_d}$ are obtained from classes of $H^{2d}_\PP((\Gamma^{d}\gl)^{(1)})$ through the map
$$\Delta_{(1,\dots,1)\,*}:H^{2d}_\PP((\Gamma^{d}\gl)^{(1)})\to H^{2d}_\PP((\otimes^{d}\gl)^{(1)})\;.$$ 
Thus, the following statement concludes the proof of theorem \ref{thm-bis} (hence the proof of theorem \ref{thm-univ}).
\begin{proposition}\label{prop-inv} Let $d\ge 2$. Then the map $\Delta_{(1,\dots,1)\,*}$ induces a surjection
$$H^{*}_\PP((\Gamma^{d}\gl)^{(1)})\twoheadrightarrow H^{*}_\PP((\otimes^{d}\gl)^{(1)})^{\Si_d}\;. $$
\end{proposition}
\begin{proof}
The image of the map $\Delta_{(1,\dots,1)\,*}$ is contained in $H^{*}_\PP((\otimes^{d}\gl)^{(1)})^{\Si_d}$, we have to show that these two graded vector spaces are equal.
By theorem \ref{thm-collapse}, we know for any bifunctor $B$ in $\PP^d_d$ a natural isomorphism
$$H^s_\PP(B_{E_r}^t)=E_\infty^{s,t}(B,1) \simeq \Gr^s H^{s+t}_\PP(B^{(1)})\;. $$
Now there is a surjective map $H^*_\PP(B^{(1)})=F^0H^{*}_\PP(B^{(1)})\twoheadrightarrow \Gr^0 H^{*}_\PP(B^{(1)})$.
In particular, we have a commutative diagram (the composites of the horizontal are the edge maps of the twisting spectral sequence)
$$\xymatrix{
H^0_\PP((\Gamma^d\gl)_{E_r} )\ar[d]^-{(2)}\ar[r]^-{\simeq}& \Gr^0 H^{*}_\PP((\Gamma^d\gl)^{(1)})\ar[d] &\ar@{->>}[l] H^{*}_\PP((\Gamma^d\gl)^{(1)})\ar[d]^-{(3)}\\ 
H^0_\PP((\otimes^d\gl)_{E_r} )^{\Si_d}\ar[r]^-{\simeq}& \Gr^0 H^{*}_\PP((\otimes^d\gl)^{(1)})^{\Si_d}&\ar[l]^-{(1)}  H^{*}_\PP((\otimes^d\gl)^{(1)})^{\Si_d} 
}.
$$
The bifunctor $\otimes^d\gl$ is injective, so we have $H^*_\PP((\otimes^d\gl)_{E_r} )=H^0_\PP((\otimes^d\gl)_{E_r} )$. Thus, the bottom right morphism $(1)$ is actually an isomorphism. The vertical map on the left $(2)$  is also an isomorphism. Indeed, $H^0_\PP(-)$ and the operation of taking invariants under the natural action of a group on a bifunctor commute up to a canonical isomorphism. Hence the vertical map on the right $(3)$ is surjective. This concludes the proof. \qed
\end{proof}

\begin{remark}
We don't know if the classes $c[d]$ are uniquely determined by the choice of $c[1]$. 
Non uniqueness is measured by the kernel of the map 
$$\Delta_{(1,\dots,1)\,*}:H^{*}_\PP((\Gamma^{d}\gl)^{(1)})\to H^{*}_\PP((\otimes^{d}\gl)^{(1)})\;.$$
The proof of proposition \ref{prop-inv} shows that this kernel is equal to $F^1 H^{*}_\PP((\Gamma^{d}\gl)^{(1)})$, which is isomorphic to the cohomology $H^{>0}_\PP((\Gamma^{d}\gl)_{E_1})$. 

This cohomology is zero if $d<p$, hence the classes $c[d]$, $2\le d\le p-1$ are uniquely determined by $c[1]$.

This cohomology is non zero if $d\ge p$. 
Indeed, for $d=p$ we can compute it explicitly by using hypercohomology spectral sequences associated to the `norm complex' 
$0\to (S^{p}\gl)_{E_1} \to (\Gamma^{p}\gl)_{E_1}\to 0$ (the map in the middle is induced by the norm map $S^p\to \Gamma^p$, which generates the dimension one vector space $\hom_\mathcal{P}(S^p,\Gamma^p)$). We find that the bigraded Poincar\'e series of $H^{>0}_\PP((\Gamma^{p}\gl)_{E_1})=\bigoplus_{s\ge 1,t\ge 0}H^{s}_\PP((\Gamma^{p}\gl)^t_{E_1})$ are equal to 
$$s\left(\frac{1-s^{2p-2}}{1-s}\right)\left(\frac{1-t^{(2p)^2}}{1-t^{2p}}\right)\;. $$
Now assume that $d>p$. The cohomology $H^{>0}_\PP(\Gamma^{p}\gl)=\bigoplus_{s\ge 1}H^{s}_\PP((\Gamma^{p}\gl)^0_{E_1})$ is non zero. The graded functor $(\Gamma^{d}\gl)_{E_1}$ contains a copy of $\Gamma^p\gl\otimes \Gamma^{d-p}\gl$ as a direct summand of its part of degree $2$, and the functor homology cup product is injective \cite[Thm. 6.1]{TouzeClassical}. So we can build nontrivial cohomology classes in $H^{>0}_\PP((\Gamma^{d}\gl)_{E_1})$ a product of classes in $H_\PP^{>0}(\Gamma^{p}\gl)$ and classes in $H_\PP^0(\Gamma^{d-p}\gl)$.

However, for $d=p$, the bigraded Poincar\'e series of $H^{>0}_\PP((\Gamma^{d}\gl)_{E_1})$ are zero in total degree $2p$. Thus, although $H^{>0}_\PP((\Gamma^{d}\gl)_{E_1})$ is far from being zero, it brings no contribution to $F^1 H^{2p}_\PP((\Gamma^{d}\gl)^{(1)})$. Hence $c[p]$ is uniquely determined by $c[1]$. Since we don't know explicitly the Poincar\'e series of $H^{>0}_\PP((\Gamma^{d}\gl)_{E_1})$ for $d>p$, we don't know if such a providential uniqueness phenomenon occurs for the higher classes.
\end{remark}

\end{document}